\documentclass{amsart}

\usepackage{amsmath}
\usepackage{amsfonts}
\usepackage{amsthm}

\newtheorem{lemma}{Lemma}
\newtheorem{prop}{Proposition}
\newtheorem{theorem}{Theorem}

\begin{document}
\title{(Non)Uniqueness of critical points in variational data assimilation}
\author{Graham Cox}
\email{ghcox@email.unc.edu}
\address{Department of Mathematics, UNC Chapel Hill, Phillips Hall CB \#3250, Chapel Hill, NC 27599}

\begin{abstract}
In this paper we apply the 4D-Var data assimilation scheme to the initialization problem for a family of quasilinear evolution equations. The resulting variational problem is non-convex, so it need not have a unique minimizer. We comment on the implications of non-uniqueness in numerical applications, then prove uniqueness results in the following situations: 1) the observational times are all sufficiently small; 2) the prior covariance is sufficiently small. We also give an example of a data set where the cost functional has a critical point of arbitrarily large Morse index.

\end{abstract}

\maketitle

%

\section{Introduction}
An important problem in data assimilation is to estimate the initial state of a physical system when only given access to noisy, incomplete observations of the state at later times. To make this more precise, suppose $y(t)$ solves an evolution equation $y_t = F(y)$ in some function space $V$, and the observations of the state are given by a bounded linear operator $H: V \rightarrow \mathbb{R}^q$. Then given observations $z_1, \ldots, z_N \in \mathbb{R}^q$ at times $t_1 <  \cdots < t_N$, one would like to find the initial condition $u = y(0)$ that best matches the empirical data.

Of course it is important to carefully formulate what is meant by the ``best" initial condition, to ensure that the problem is well-posed and has a physically meaningful solution. One approach is to minimize the log-likelihood
\begin{align*}
	\sum_{i=1}^N \left| R^{-1/2} \left(H y(t_i) - z_i \right) \right|^2
\end{align*}
over the set of all possible initial conditions $u$, where $R$ is the observational covariance matrix, and $y(t_i)$ is the solution to the evolution equation with initial condition $u$, evaluated at time $t_i$. However, the resulting variational problem turns out to be ill-posed, in the sense that it does not necessarily possess a minimizer in $V$.

One possible resolution is to add a regularization term to the cost functional, of the form
\begin{align}
	J(u) := \frac{1}{2} \sum_{i=1}^N \left| R^{-1/2} \left(H y(t_i) - z_i \right) \right|^2 + \frac{1}{2 \sigma^2} \| u - u_0\|_V^2
	\label{eqn:cost}
\end{align}
for some fixed $u_0 \in V$ and $\sigma > 0$. The analytic motivation for this is clear---the cost functional is now coercive over $V$ and hence can be shown through standard variational methods to admit a minimizer (see \cite{S10} and \cite{W93} for details). From a data analysis point of view, there is a Bayesian interpretation of (\ref{eqn:cost}) in which the regularization term corresponds to a Gaussian prior distribution with covariance proportional to $\sigma^2$. 

It is common practice (see, for instance, \cite{CPZ08,DT86,LKW06,NW07}) to solve a suitable discretization of the regularized variational problem using a gradient-based algorithm. Implicit in the implementation of such an algorithm is the assumption of a unique minimizer for the variational problem---gradient descent methods are of course local and do not have the ability to distinguish between local and global minima. The problem of uniqueness has so far received little attention in the literature. A short-time uniqueness result for Burgers' equation appeared in \cite{W93} under the assumption of continuous-in-time observations, using the cost functional
\begin{align*}
	\int_0^T \left| H y(t) - z(t) \right|^2 dt + \frac{1}{2 \sigma^2} \| u - u_0\|_V^2.
\end{align*}
There it was shown that the variational problem admits a unique minimizer when the maximal observation time, $T$, is sufficiently small. 

The discrete problem was investigated numerically in \cite{AAR10}, where a unique minimizer was observed as long as $\sigma > 0$. For the non-regularized $\sigma=0$ case (corresponding to an improper prior in the Bayesian formulation) multiple minimizers were found numerically.

The goal of this paper is to give a rigorous Bayesian formulation of the variational problem for a family of quasilinear evolution equations (which includes reaction-diffusion equations and viscous conservation laws) and determine sufficient conditions to guarantee unimodality of the resulting posterior distribution.

\subsection{Some notation and conventions}
Throughout we denote the $L^2(0,1)$ norm and inner product by $\| \cdot \|$ and $\left< \cdot, \cdot \right>$, respectively. We let $V = H^1_0(0,1)$, with norm $\|u\|_V := \|u_x\|$. This is clearly equivalent to the standard $H^1$ norm, because $\pi \|u\| \leq \|u_x\|$ for any $u \in H^1(0,1)$. It is well known that $H^1_0(0,1) \subset L^{\infty}(0,1)$, with $\sup |u| \leq \|u_x\|$. We will frequently make use of the inequality between the arithmetic and geometric means,
\begin{align}
	2ab \leq \lambda a^2 + \lambda^{-1} b^2
\end{align}
for any positive $a$, $b$ and $\lambda$, which we refer to as the AM--GM inequality.


\section{Statement of results}
For the remainder of the paper we consider a quasilinear parabolic equation
\begin{align}
	y_t + f(y)_x = y_{xx} + r(y)
	\label{eqn:evol}
\end{align}
on the interval $[0,1]$, with Dirichlet boundary conditions. We make the standing assumption that $f$ and $r$ are both of class $C^2$. This is more than sufficient to guarantee that the initial value problem for (\ref{eqn:evol}) is well-posed, as will be seen in Proposition \ref{prop:exist}. The additional regularity is needed in computing the first and second variation of the cost functional. We also need to ensure that the initial value problem admits a global (in time) solution for any initial value, so that $J$ is well-defined on all of $H^1_0$. This will be the case if
\begin{align}
	\int_{-\infty}^0 \frac{1}{|r(y)|+1} dy = \int_0^{\infty} \frac{1}{|r(y)|+1} dy = \infty.
	\label{growth}
\end{align}
If this condition is not satisfied, there may exist initial conditions for which the solution blows up in a finite amount of time. 

We also assume that the observation operator $H$ is bounded on $L^2$, and hence has a bounded adjoint $H^*: \mathbb{R}^q \rightarrow L^2$.

Our first result is that the problem has a natural Bayesian formulation with respect to a Gaussian prior distribution, the significance of which will be discussed in Section \ref{sec:bayes}. This requires a further assumption on $r$ and $f$ that will not be needed elsewhere in the paper.

\begin{theorem} Let $\mu_0$ denote the Gaussian measure on $L^2(0,1)$ with covariance $\mathcal{C} = -\sigma^2 \Delta^{-1}$ and mean $u_0$, and suppose that $r(y)$ and $f'(y)$ are uniformly Lipschitz. Then there is a well-defined posterior measure $\mu_z$, with Radon--Nikodym derivative
\begin{align}
	\frac{d \mu_z}{d \mu_0}(u) \propto \exp \left\{ - \sum_{i=1}^N \left| R^{-1/2} \left(H y(t_i) - z_i \right) \right|^2 \right\}.
	\label{eqn:RNderiv}
\end{align}
Moreover, the mean and covariance of the posterior distribution are continuous functions of the data, $z = \{z_i\}$.
\label{thm:bayes}
\end{theorem}

In fact, one has that the posterior measure is Lipschitz with respect to the Hellinger metric; the reader is referred to \cite{S10} for further details. The nontriviality of this result is due to the infinite-dimensional setting of the problem. Because there is no analog of the Lebesgue measure for infinite-dimensional spaces, one cannot define the posterior measure using the exponential of $J$ as a density, as is done in finite dimensions. Thus it is necessary to define the posterior relative to the prior distribution, and care must be taken to ensure that this density, given by (\ref{eqn:RNderiv}), is in fact $\mu_0$-integrable and hence can be normalized. This normalizability will follow from estimates on solutions to the nonlinear evolution equation.

There is thus a Bayesian formulation of the regularized variational problem, for which the MAP (Maximum \textit{A Posteriori}) estimators are precisely the global minima of the cost functional (\ref{eqn:cost}). With this framework in mind, we study the uniqueness and non-uniqueness of minima for $J(u)$.

We assume throughout that the data are uniformly bounded, with
\begin{align}
	|z_i| \leq D
\end{align}
for all $i$. Our first result is that $J$ has a unique minimizer when all of the observational times are sufficiently small. 
\begin{theorem} There is a constant $T_0$, depending on $N$, $D$, $\|u_0\|$ and $\sigma$, such that (\ref{eqn:cost}) has a unique global minimum in $V$ if $t_N < T_0$.
\label{thm:shorttime}
\end{theorem}

The time $T_0$ also depends on the observation operator, $H$ and the observational covariance, $R$, but we consider these to be fixed throughout, and hence will not explicitly note this dependence. We will similarly not mention any dependence of constants on the functions $f$ and $r$ in (\ref{eqn:evol}), though this dependence can easily be deduced from the proofs if desired. 

The theorem is proved in Section \ref{sec:unproofs} by first observing that all minimizers are contained in a fixed ball $B \subset V$, then showing that the cost functional is convex over $B$ as long as the observational times are small enough that nonlinear effects are not yet dominant. This differs from the uniqueness result in \cite{W93} because in the discrete-time case there are non-vanishing contributions to the cost functional even as $t_N \rightarrow 0$, whereas in the continuous case the observational term
\begin{align*}
	\int_0^T \left| H y(t) - z(t) \right|^2 dt
\end{align*}
vanishes in the $T = 0$ limit. For this reason we need to consider the second variation of the cost functional. In the continuous case the Euler--Lagrange equation can be expressed as a fixed-point equation for a nonlinear map that is a contraction for small $T$, but this contractive property is easily seen to fail in the discrete case, even for linear equations.




We next show that it is possible to obtain a uniqueness result for any set of observational times, provided the observational covariance is sufficiently small.

\begin{theorem} There is a constant $\sigma_0 > 0$, depending on, $N$, $D$, $\|u_0\|$ and $t_N$, such that (\ref{eqn:cost}) has a unique global minimum in $V$ if $\sigma < \sigma_0$.
\label{thm:alltime}
\end{theorem}

We will see explicitly in (\ref{eqn:2var}) how $f''$ and $r''$ can lead to nonconvexity in $J$. The general idea behind the preceding uniqueness theorems is thus to determine under what conditions these nonlinear effects can be dominated by the linear term coming from the Gaussian prior distribution.

It will be seen in the proofs below that Theorems \ref{thm:shorttime} and \ref{thm:alltime} (as well as the uniqueness result in \cite{W93}) in fact establish the stronger result that the cost functional has a unique critical point in a closed subset of $H^1_0$ that necessarily contains any global minima. This  observation could be useful in implementing a gradient descent method, because it says that one can avoid suprious local minima by ensuring that the algorithm starts in the bounded region given by Lemma \ref{lem:apriori}, where $J$ is known to be convex.

Our final result shows that the behavior of $J$ can be rather complicated in general.

\begin{theorem} Consider a reaction-diffusion equation $y_t = y_{xx} + r(y)$ where $r(0) = r'(0) = 0$ and $r''(0) \neq 0$. Let $H$ denote projection onto the first Fourier coefficient, and set $R = 1$. Then for any positive integer $q$ and times $t_1 < \cdots < t_N$ there exist data $\{z_i\}$ and a prior $u_0$ such that $u \equiv 0$ is a critical point of $J$ with Morse index greater than or equal to $q$.
\label{thm:non}
\end{theorem}

Thus there are cases in which $J$ is not globally convex, and has at least two critical points (since it is already known to have a minimizer).


\section{The variational framework}
We start our investigation by deriving the Euler--Lagrange equation for the variational problem (\ref{eqn:cost}) in the space $V$. We also compute the second variation of the cost functional as it will be needed in proving the uniqueness theorems.

We first recall that $y$ denotes the unique solution to (\ref{eqn:evol}) with Dirichlet boundary conditions and $y(0) = u$. The variation of $y$ with respect to the initial value $u$, in the $v$-direction, is denoted $\eta := Dy(u)v$, and satisfies the initial value problem
\begin{align}
	\eta_t + \left[f'(y) \eta \right]_x = \eta_{xx} + r'(y) \eta \label{eqn:linear} \\
	\eta(0) = v. \nonumber
\end{align}
Similarly, the second variation of $y$ is denoted $\omega := D^2y(u)(v,v)$ and satisfies
\begin{align}
	\omega_t + \left[f'(y) \omega + f''(y) \eta^2 \right]_x = \omega_{xx} + r'(y) \omega + r''(y) \eta^2 \label{eqn:quad} \\
	\omega(0) = 0. \nonumber
\end{align}
We observe that $\omega \equiv 0$ if $f''$ and $r''$ vanish, which happens precisely when the forward equation is linear.

We let $p$ denote the solution to the adjoint equation
\begin{align}
	-p_t - f'(y) p_x = p_{xx} + r'(y) p \label{eqn:adjoint} \\
	p(t_N) = 0 \nonumber
\end{align}
with a discontinuous jump
\begin{align}
	p(t_i^+) - p(t_i^-) = H^* R^{-1} \left( H y(t_i) - z_i \right)
	\label{eqn:jump}
\end{align}
prescribed at each observation time, $t_i$. The definition is such that
\begin{align}
	\left< p_t, \eta \right> + \left< p, \eta_t \right> = 0
	\label{eqn:peta}
\end{align}
for all $t \neq t_i$.

The first variation of the cost functional (\ref{eqn:cost}) can thus be written
\begin{align*}
	DJ(u)(v) = \sum_{i=1}^N \left< \eta(t_i), H^* R^{-1} (H y(t_i) - z_i) \right> + \frac{1}{\sigma^2} \left<v, u-u_0 \right>_V,
\end{align*}
and from the definition of $p$ we obtain
\begin{align*}
	\sum_{i=1}^N \left< \eta(t_i), H^* R^{-1} (H y(t_i) - z_i) \right> = &\sum_{i=1}^N \left<\eta(t_i), p(t_i^+) - p(t_i^-) \right> \\
	= &-\left<\eta(0), p(0) \right> \\ 
	&- \sum_{i=1}^N \left[ \left< \eta(t_i), p(t_i^-) \right> - \left< \eta(t_{i-1}), p(t_{i-1}^+) \right> \right]
\end{align*}
where we have set $t_0 = 0$. Then (\ref{eqn:peta}) implies
\begin{align*}
	\left< \eta(t_i), p(t_i^-) \right> - \left< \eta(t_{i-1}), p(t_{i-1}^+) \right> &= \int_{t_{i-1}}^{t_i} \frac{d}{dt} \left< p, \eta \right> dt \\
	&= 0
\end{align*}
for each $i$, hence
\begin{align*}
	DJ(u)v &= - \left< v, p(0) \right> + \sigma^{-2} \left<v, u - u_0 \right>_V.
\end{align*}
Integrating the first term by parts, we arrive at the following.

\begin{prop} The $V$-gradient of $J$ is given by
\begin{align}
	DJ(u) &= \Delta^{-1} p(0) + \sigma^{-2} (u - u_0).
	\label{eqn:EL}
\end{align}
\end{prop}
To better understand this result, it is worth recalling that the solution to the adjoint equation depends on $y$ and hence on the initial condition, $u$. With this dependence explicitly written as $p[y(u)]$, the Euler--Lagrange equation can be viewed as a fixed-point equation for the map
\begin{align}
	u \mapsto u_0 - \sigma^2 \Delta^{-1} p[y(u)] (0)
	\label{map:EL}
\end{align}



Proceeding similarly for the second variation, we find
\begin{align}
	D^2J(u)(v,v) &= \sum_{i=1}^N \left[ \left< \omega(t_i), H^* R^{-1} (H y(t_i) - z_i) \right> + \|R^{-1/2} H \eta(t_i) \|^2 \right] + \frac{1}{\sigma^2} \| v  \|^2_V
\end{align}
and
\begin{align*}
	\sum_{i=1}^N \left< \omega(t_i), H^* R^{-1} (H y(t_i) - z_i) \right> =
	- \sum_{i=1}^N \left[ \left< \omega(t_i), p(t_i^-) \right> - \left< \omega(t_{i-1}), p(t_{i-1}^+) \right> \right]
\end{align*}
because $\omega(0) = 0$. Using (\ref{eqn:quad}) and (\ref{eqn:adjoint}) and integrating by parts, we find that
\begin{align*}
	\left< p_t, \omega \right> + \left< p, \omega_t \right> = \left< r''(y) \eta^2, p \right> + \left< f''(y) \eta^2, p_x \right>,
\end{align*}
with the following consequence.
\begin{prop} The Hessian of $J$ is given by
\begin{align}
	\label{eqn:2var}
	D^2J(u)(v,v) = &\int_0^{t_N} \left[ \left< r''(y) \eta^2, p \right> + \left< f''(y) \eta^2, p_x \right> \right] dt \\ &+ \sum_{i=1}^N \left|R^{-1/2} H \eta(t_i) \right|^2 + \frac{1}{\sigma^2} \| v  \|^2_V. \nonumber
\end{align}
\end{prop}

We observe that this term is positive definite if $f'' =r'' = 0$, but the Euler--Lagrange map given in (\ref{map:EL}) may fail to be a contraction in that case.


\section{Analytic preliminaries}
\label{sec:anal}
In this section we review some analysis for quasilinear evolution equations. We restrict our attention to equations that admit global solutions for all initial conditions, as this guarantees the cost functional is defined on all of $V$.

\begin{prop} Suppose $r$ and $f'$ are locally Lipschitz, and (\ref{growth}) is satisfied. Then for any initial condition $u \in H^1_0$, (\ref{eqn:evol}) has a unique classical solution on $[0,1] \times (0,\infty)$.
\label{prop:exist}
\end{prop}

This result is a direct consequence of the material in Chapter 3 of \cite{H81}; for the sake of completeness we verify some of the necessary details here.

\begin{proof} The local existence and uniqueness follows from Theorem 3.3.3 of \cite{H81}, because the map $y \mapsto r(y) - f'(y) y_x$ is locally Lipschitz from $H^1_0$ to $L^2$. This yields a solution $y \in C\left( [0,T); L^2 \right)$, with $y(t) \in H^1_0 \cap H^2$ for all $t \in (0,T)$. Moreover, from Theorem 3.5.2 of \cite{H81}, the map $t \mapsto y_t \in H^1_0$ is locally H\"{o}lder continuous. In particular this implies $y$ and $y_t$ are continuous in both $x$ and $t$. We also know that $y \in H^2$ so by the Sobolev embedding theorem $y_x \in H^1$ is H\"{o}lder continuous. We then have for each fixed $t$ that
\begin{align*}
	y_{xx} = y_t + f(y)_x - r(y)
\end{align*}
is in $C^{\delta}[0,1]$ for some $\delta > 0$, hence by elliptic regularity $y(t) \in C^{2+\delta}[0,1]$. Thus $y$ is a classical solution of (\ref{eqn:evol}).

The long-time existence claim follows from Corollary 3.3.5 of \cite{H81} together with the following pointwise bound from Lemma \ref{yunif}.

\end{proof}

We now gather some estimates on $y$, $\eta$ and $p$ that will be needed in proving the uniqueness theorems. The first of these shows that $y$ is uniformly bounded on any finite time interval.

\begin{lemma} Suppose $y(t)$ solves (\ref{eqn:evol}) for $t \in (0,T)$, with $\|y(0)\|_{\infty} \leq A$. Then
\begin{align}
	\|y(t)\|_{\infty} \leq B
\end{align}
for any $t<T$, where $B$ depends on $A$ and $T$.
\label{yunif}
\end{lemma}

\begin{proof}
Let $\psi_+$ and $\psi_-$ solve the initial value problems
\begin{align*}
	\psi_+' = |r(\psi^+)|, &\ \psi_+(0) = A \\
	\psi_-' = -|r(\psi^-)|, &\ \psi_-(0) = -A.
\end{align*}
Then the parabolic maximum principle (\textit{cf.} Theorem 1 of \cite{K63}) guarantees that
\begin{align*}
	\psi_-(t) \leq y(x,t) \leq \psi_+(t).
\end{align*}
From (\ref{growth}) we know that $\psi_{\pm}$ exist and are continuous for all $t \geq 0$, so we complete the proof by setting
\begin{align*}
	B := \max_{0 \leq t \leq T} \max \left\{ -\psi_-(t), \psi_+(t) \right\}
\end{align*}
\end{proof}

Since $f$ and $r$ are of class $C^2$, for any given value of $A$ and $k \in \{0,1,2\}$ the quantities
\begin{align}
	R_k(A) &:= \sup_{|y| \leq B} |r^{(k)}(y)| \\
	F_k(A) &:= \sup_{|y| \leq B} |f^{(k)}(y)|
\end{align}
are well defined, with $B$ as in the proof of Lemma \ref{yunif}.

We next derive an estimate on the $L^4$ norm of $\eta$. To simplify notation, we observe that $\| \eta \|_{L^4(0,1)} = \| \eta^2 \|^{1/2}$.
\begin{lemma} Suppose $\eta(t)$ solves (\ref{eqn:linear}) for $t \in (0,T)$, and $\|y(0)\|_{\infty} \leq A$. Then
\begin{align}
	\| \eta^2(t) \| \leq \| v^2 \| e^{\alpha t}
\end{align}
for any $t <  T$, where $\alpha$ depends on $A$ and $T$.
\label{etaL4}
\end{lemma}

\begin{proof} We differentiate and then integrate by parts to obtain
\begin{align*}
	\frac{1}{4} \frac{d}{dt} \int_0^1 \eta ^4 dx &= \int_0^1 \eta^3 \left( \eta_{xx} - [ f'(y) \eta]_x + r'(y) \eta \right) dx \\
	&= \int_0^1 \left( -3 \eta^2 \eta_x^2 + 3 f'(y) \eta^3 \eta_x + r'(y) \eta^4 \right) dx
\end{align*}
Then by the AM--GM inequality
\begin{align*}
	\left| \eta^3 \eta_x \right| &\leq \frac{F_1(A)}{4} \eta^4 + \frac{1}{F_1(A)} \eta^2 \eta_x^2
\end{align*}
so we find that
\begin{align*}
	\frac{d}{dt} \int_0^1 \eta ^4 dx & \leq \left(4 R_1(A) + 3F_1(A)^2 \right) \int_0^1 \eta^4 dx.
\end{align*}
The result follows from Gronwall's inequality.
\end{proof}

We finally turn to the adjoint equation. Invoking linearity, the solution can be expressed as $p = p_1 + \cdots p_N$, where $p_i$ satisfies (\ref{eqn:adjoint}) with terminal condition $p(t_N) = 0$ and jump
\begin{align}
	p(t_i+) - p(t_i-) = H^* R^{-1} (H y(t_i) - z_i).
\end{align}
Therefore it suffices to bound each $p_i$ individually, then sum the resulting estimates.

\begin{lemma} Suppose $p(t)$ solves (\ref{eqn:adjoint}), and $\|y(0)\|_{\infty} \leq A$. Then
\begin{align}
	\| p(t) \| \leq C e^{\beta (t_N-t)}
	\label{pbound}
\end{align}
for any $t  \leq t_N$, and
\begin{align}
	\int_0^{t_N} \|  p_x(t) \| dt \leq C \sqrt{t_N} e^{2\beta t_N},
	\label{pxbound}
\end{align}
where $\beta$ depends on $A$ and $t_N$, and $C$ depends on $N$, $D$, $A$ and $t_N$.
\label{pL2}
\end{lemma}


\begin{proof} Differentiating and applying the AM--GM inequality, as in the proof of Lemma \ref{etaL4}, we have
\begin{align*}
	-\frac{1}{2} \frac{d}{dt} \| p \|^2 &= \left<p, p_{xx} + r'(y)p + f'(y) p_x \right> \\
	&\leq - \| p_x \|^2 + R_1(A) \|p\|^2 + F_1(A) \|p\| \|p_x\| \\
	&\leq \left(R_1(A) + \frac{F_1(A)^2}{4} \right) \|p\|^2
\end{align*}
and so an application of Gronwall's inequality to the function $\| p(t_i-t) \|^2$ yields
\begin{align*}
	\| p_i(t) \| \leq \left\| H^* R^{-1} (H y(t_i) - z_i) \right\| e^{\beta(t_i-t)}
\end{align*}
for any $t \leq t_i$, with $\beta = R_1(A) + F_1(A)^2/4$. We next recall that $y(t)$ is bounded uniformly (and hence in $L^2$), so
\begin{align*}
	\left\| H^* R^{-1} (H y(t_i) - z_i) \right\| \leq C',
\end{align*}
where $C'$ depends on $A$ and $t_N$ (through Lemma \ref{yunif}) and $D$. To complete the proof of (\ref{pbound}) we simply note that $p_i(t) = 0$ for $t > t_i$, then let $C = NC'$.

With a different choice of constants in the AM--GM inequality, we obtain
\begin{align*}
	-\frac{1}{2} \frac{d}{dt} \| p \|^2 \leq - \frac{1}{2} \| p_x \|^2 + \left(R_1(A) + \frac{F_1(A)^2}{2} \right) \|p\|^2
\end{align*}
and subsequently, letting $\gamma = 2R_1(A) + F_1(A)^2$,
\begin{align*}
	\| p_x \|^2 
	& \leq \frac{d}{dt} \left( e^{\gamma t} \|p\|^2 \right).
\end{align*}
Integrating, we find
\begin{align*}
	\int_0^{t_i} \| p_{ix}(t) \|^2 dt &\leq e^{\gamma t_i} \|p_i(t_i) \|^2 - \|p_i(0)\|^2 \\
	& \leq e^{\gamma t_i} C'^2
\end{align*}
for each $i$.
Now from the Cauchy--Schwarz inequality,
\begin{align*}
	\int_0^{t_N} \|p_x(t)\| dt 
	&\leq \sum_{i=1}^N \left( t_i \int_0^{t_i} \|p_{xi}(t)\|^2 dt  \right)^{1/2} \\
	&\leq N C' \sqrt{t_N} e^{\beta t_N},
\end{align*}
where we have used the fact that $\gamma \leq 4 \beta$.
\end{proof}


\section{The Bayesian formulation}
\label{sec:bayes}
Before proving Theorem \ref{thm:bayes} we elaborate on the meaning of the Gaussian prior measure $\mu_0 =  \mathcal{N} \left(u_0, -\sigma^2 \Delta^{-1} \right)$, following throughout the presentation of \cite{S10}.

We first note that the covariance operator $\mathcal{C} = -\sigma^2 \Delta^{-1}$ has eigenvalues $\gamma_n = (\sigma / n\pi)^2$, with normalized eigenfunctions $\phi_n(x) = \sqrt{2} \sin(n\pi x)$. Then we can express a random variable $u \sim \mu_0$ using the Karhunen--Lo\`{e}ve expansion:
\begin{align}
	u = u_0 + \sqrt{2} \sum_{n=1}^{\infty} \frac{\sigma \xi_n}{n \pi} \sin(n\pi x),
\end{align}
where $\{ \xi_n\}$ is an i.i.d. sequence of $ \mathcal{N}(0,1)$ random variables. This means the $n^{\textrm{th}}$ Fourier coefficient of $u$ is distributed according to $\mathcal{N}(a_n, (\sigma/n\pi)^2)$, where $a_n$ is the $n^{\textrm{th}}$ Fourier coefficient of the prior mean, $u_0$. It follows that
\begin{align*}
	\|u - u_0\|^2 &=  \sum_{n=1}^{\infty} \frac{\sigma^2 \xi_n^2}{(n \pi)^2}
\end{align*}
and so
\begin{align*}
	\mathbb{E} \|u - u_0\|^2 = \frac{\sigma^2}{6}.
\end{align*}
Thus $\sigma^2$ measures the expected value of $\|u-u_0\|^2$.


We now observe that Theorem \ref{thm:bayes} follows from Corollary 4.4 of \cite{S10}. To do so we must show that:
\begin{enumerate}
	\item[(i)] $L^2(0,1)$ has full measure under $\mu_0$;
	\item[(ii)] for every $\epsilon > 0$ there exists $M \in \mathbb{R}$ such that
		\begin{align*}
			\sum_{i=1}^N \left| R^{-1/2} H y(t_i) \right|^2 \leq \exp(\epsilon \|y(0)\|^2 + M)
		\end{align*}
	whenever $y(t)$ is a solution to (\ref{eqn:evol});
	\item[(iii)] for every $\rho > 0$ there exists $L \in \mathbb{R}$ such that
		\begin{align*}
			\sum_{i=1}^N \left| R^{-1/2} H \left(y_1(t_i) - y_2(t_i) \right) \right|^2 \leq L \|y_1(0) - y_2(0)\|^2
		\end{align*}
	whenever $y_1(t)$ and $y_2(t)$ satisfy (\ref{eqn:evol}) with $\max\{ \|y_1(0)\|, \|y_2(0)\| \} < \rho$.
\end{enumerate}

To establish (i) we use Lemma 6.25 of \cite{S10}, which says that any function $u \sim \mu_0$ is almost surely $\alpha$-H\"{o}lder continuous for any $\alpha < 1/2$. In particular, this implies $u$ is almost surely contained in $L^2(0,1)$, hence $\mu_0[L^2(0,1)] = 1$.

Since the observation operator $H$ is bounded on $L^2$, (ii) and (iii) will follow from Lemmas \ref{lem:yL2} and \ref{lem:ylip} below.

\begin{lemma} Suppose $r(y)$ is uniformly Lipschitz. Then there exist positive constants $a$ and $b$ so that
\begin{align}
	\|y(t)\|^2 \leq e^{at} \left[ \|y(0)\|^2 + bt \right]
	\label{eqn:yL2}
\end{align}
and
\begin{align}
	2\int_0^t \|y_x(s)\|^2 ds \leq \|y(0)\|^2 - \|y(t)\|^2 + a \int_0^t \|y(s)\|^2 ds + bt
	\label{eqn:yx}
\end{align}
for all $t \geq 0$ and any solution $y(t)$ to (\ref{eqn:evol}).
\label{lem:yL2}
\end{lemma}

\begin{proof} 
Differentiating, we have
\begin{align*}
	\frac{1}{2} \frac{d}{dt} \| y \|^2 &= \left< y, y_{xx} + r(y) - f(y)_x  \right>.
\end{align*}
Letting $g(y)$ be an antiderivative of $yf'(y)$, we find that
\begin{align*}
	\left<y, f(y)_x \right> 
	&= \int_0^1 g(y)_x dx
\end{align*}
vanishes by the fundamental theorem of calculus, so
\begin{align*}
	\frac{d}{dt} \| y \|^2 &\leq -\|y_x\|^2 + 2\left< y, r(y) \right>.
\end{align*}
It follows immediately from the Lipschitz condition that $|yr(y)| \leq K |y|^2 + |r(0)| |y|$, which implies $2 |y r(y)| \leq a|y|^2 + b$ for some $a$ and $b$. Then (\ref{eqn:yL2}) is a consequence of Gronwall's inequality, and (\ref{eqn:yx}) is obtained by integrating from $0$ to $t$.
\end{proof}

To see how this implies (ii), we first observe that it suffices to prove
\begin{align*}
	\| y(t) \|^2 \leq \exp(\epsilon \|u\|^2 + M)
\end{align*}
for any $t \leq t_N$. Thus fixing $\epsilon > 0$ and defining
\begin{align*}
	e^M = \max\{ b t_N e^{a t_N}, \epsilon^{-1} e^{a t_N} \}
\end{align*}
we have from Lemma \ref{lem:yL2} that
\begin{align*}
	\|y(t)\|^2 &\leq e^M \left(1 + \epsilon \|u\|^2 \right) \\
	& \leq e^M e^{\epsilon \|u\|^2}
\end{align*}
as required.

\begin{lemma}  Suppose $r(y)$ and $f'(y)$ are uniformly Lipschitz. Then for any $\rho > 0$ there exists a positive constant $L$ so that
\begin{align*}
	\|y_1(t) - y_2(t)\|^2 \leq L \|u_1 - u_2\|^2
\end{align*}
for all $t \leq t_N$, provided $y_1(t)$ and $y_2(t)$ satisfy (\ref{eqn:evol}) with $\max \{ \|y_1(0)\|, \|y_2(0)\| \} < \rho$.
\label{lem:ylip}
\end{lemma}

\begin{proof} For convenience we let $K_r$ and $K_f$ denote the Lipschitz constants of $r$ and $f'$, respectively. Differentiating and then integrating by parts as in the proof of Lemma \ref{lem:yL2}, we have
\begin{align*}
	\frac{1}{2} \frac{d}{dt} \| y_1 - y_2 \|^2 &= \left< y_1 - y_2, (y_1-y_2)_{xx} + r(y_1) - r(y_2) - f(y_1)_x + f(y_2)_x  \right> \\
	&\leq -\|(y_1-y_2)_x\|^2 + K_r \|y_1 - y_2\|^2 +  \left< y_1 - y_2, f(y_2)_x - f(y_1)_x  \right>.
\end{align*}
For the final term we write
\begin{align*}
	f(y_1)_x - f(y_2)_x = f'(y_1) (y_1 - y_2)_x + \left[f'(y_1) - f'(y_2) \right] y_{2x}.
\end{align*}
and thus obtain
\begin{align*}
	\left| \left< y_1 - y_2, f(y_2)_x - f(y_1)_x  \right> \right|  \leq &\left[ K_f \left( \| y_{1x} \| + \|y_{2x} \| \right) + |f'(0)| \right] \|y_1 - y_2\| \| (y_1 - y_2)_x\|.
\end{align*}
Then after an application of the AM--GM inequality, we find that
\begin{align*}
	\frac{1}{2} \frac{d}{dt} \| y_1 - y_2 \|^2 &\leq \left( K_r + \frac{1}{4} \left[ K_f \left( \| y_{1x} \| + \|y_{2x} \| \right) + |f'(0)| \right]^2 \right) \|y_1 - y_2\|^2 \\
	&= \alpha(t) \|y_1 - y_2\|^2,
\end{align*}
where we have defined $\alpha(t)$ to be the term in parentheses on the right-hand side. From (\ref{eqn:yx}) we know that $\int_0^t \alpha(s) ds$ is bounded above by a constant depending only on $\|y_1(0)\|$, $\|y_2(0)\|$ and $t_N$, so the result follows from Gronwall's inequality.
\end{proof}

\section{The uniqueness theorems}
\label{sec:unproofs}
Our main tool for proving Theorems \ref{thm:shorttime} and \ref{thm:alltime} will be the second variation formula (\ref{eqn:2var}) together with the following \textit{a priori} estimate for minimizers of $J$.

\begin{lemma} Let $u^*$ achieve of the infimum of the cost functional (\ref{eqn:cost}). Then
\begin{align}
	\|u^*\|_V \leq A,
\end{align}
where $A$ depends on $N$, $D$, $t_N$, $\|u_0\|$ and $\sigma$.
\label{lem:apriori}
\end{lemma}

It is clear from the proof that $A$ can be assumed to be nondecreasing with respect to $\sigma$.

\begin{proof}
Since $u^*$ is a minimizer it satisfies $J(u^*) \leq J(0)$. Letting $y(t)$ solve (\ref{eqn:evol}) with $y(0) = 0$, Lemma \ref{yunif} implies that $y(t)$ is uniformly bounded for $t \leq t_N$. Therefore
\begin{align*}
	\| u^* \|_V^2 & \leq 2 \sigma^2 J(u^*) \\
	& \leq \sigma^2 \sum_{i=1}^N \left| R^{-1/2} \left( H y(t_i) - z_i\right) \right|^2 +  \| u_0 \|_V^2
\end{align*}
is bounded above as claimed.
\end{proof}

We now use the estimates of Section \ref{sec:anal} to prove that, under the conditions of Theorems \ref{thm:shorttime} and \ref{thm:alltime}, $J$ is convex on the ball $\|u^*\|_V \leq A$. Discarding nonnegative terms in (\ref{eqn:2var}), it suffices to show that
\begin{align*}
	\int_0^{t_N} \left| \left< r''(y) \eta^2, p \right> + \left< f''(y) \eta^2, p_x \right> \right| dt < \frac{1}{\sigma^2} \| v  \|^2_V.
\end{align*}
From Lemmas \ref{etaL4} and \ref{pL2} we have
\begin{align*}
	\left| \left< r''(y) \eta^2, p \right> \right| \leq C R_2 \|v^2\| e^{\alpha t+\beta(t_N-t)}
\end{align*}
and
\begin{align*}
	\int_0^{t_N} \left| \left< f''(y) \eta^2, p_x \right> \right| dt \leq C F_2 \|v^2\| \sqrt{t_N} e^{(\alpha+2\beta) t_N}.
\end{align*}
Combining these estimates, we find that
\begin{align}
	\int_0^{t_N} \left| \left< r''(y) \eta^2, p \right> + \left< f''(y) \eta^2, p_x \right> \right| dt \leq \Gamma \|v\|_V^2 \sqrt{t_N},
	\label{2estimate}
\end{align}
where $\Gamma$ depends on $A$ (from Lemma \ref{lem:apriori}), $t_N$, $N$ and $D$.

It is clear that the constant $\Gamma$ in (\ref{2estimate}) remains bounded as $t_N \rightarrow 0$. Therefore in proving Theorem \ref{thm:shorttime} it suffices to choose $t_N$ sufficiently small that $\Gamma \sqrt{t_N} < \sigma^{-2}$. Similarly for Theorem \ref{thm:alltime}, we observe that $\Gamma$ remains bounded as $\sigma \rightarrow 0$ so it is possible to choose $\sigma$ small enough that $\sigma^{-2} > \Gamma \sqrt{t_N}$.


\section{The non-uniqueness theorem}
Turning now to the proof of Theorem \ref{thm:non}, we must establish that $u = 0$ is a critical point of $J$, and the Hessian $D^2J(0)$ has at least $q$ negative eigenvalues. The key to the proof is the observation that the Euler--Lagrange equation depends on both the data and the prior, whereas the Hessian is independent of the prior. Thus we can first construct data to ensure $D^2 J(0)$ has the required number of negative eigenvalues, and then choose the prior term to ensure that $0$ is in fact a critical point of $J$.

The hypothesis $r(0) = 0$ ensures that $y(t)=0$ is the unique solution of (\ref{eqn:evol}) with $y(0) = 0$. Then because $r'(0) = 0$, the linearized equation reduces to the heat equation, $\eta_t = \eta_{xx}$, and the adjoint equation becomes the backward heat equation, $-p_t = p_{xx}$.

We compute the Hessian of $J$ in the direction of the first $q$ Fourier modes, setting $v_n = \sin(n \pi x)$ for $1 \leq n \leq q$, so $\|v_n\|^2_V = 1 / 2$. The corresponding solution to the linearized forward equation is
\begin{align*}
	\eta(x,t) = e^{-n^2 \pi^2 t}  \sin(n\pi x)
\end{align*}
and so
\begin{align*}
	\sum_{i=1}^N \left|R^{-1/2} H \eta(t_i) \right|^2  \leq \sum_{i=1}^N e^{-2 \pi^2 t_i}
\end{align*}
For each observation $z_i \in \mathbb{R}$ we have
\begin{align*}
	H^* z_i = z_i \sin(\pi x),
\end{align*}
hence the solution to the adjoint equation is given by
\begin{align*}
	p(x,t) = \sum_{ \{i: t < t_i \}}  z_i e^{-\pi^2(t-t_i)} \sin(\pi x)
\end{align*}
for $t \neq t_i$. We thus find that
\begin{align*}
	\left< r''(y) \eta^2, p \right> &= \sum_{ \{i: t < t_i \}} z_i e^{-2 n^2 \pi^2 t_i} e^{-\pi^2(t-t_i)} \left< \sin(\pi x), \sin^2(n \pi x) \right> \\
	&= \frac{4 n^2}{2\pi(4n^2-1)} \sum_{ \{i: t < t_i \}} z_i e^{-\pi^2 [t + (2 n^2 -1) t_i]}
\end{align*}
Integrating, we have
\begin{align*}
	\int_0^{t_N} \sum_{ \{i: t < t_i \}} z_i e^{-\pi^2 [t + (2 n^2 -1) t_i]} dt
	&=  \frac{1}{\pi^2} \sum_{i=1}^N z_i e^{- 2 n^2 \pi^2 t_i} \left( e^{\pi^2 t_i}  - 1 \right)
\end{align*}
and so we find from (\ref{eqn:2var}) that
\begin{align*}
	D^2 J(0)(v_n,v_n) \leq \frac{4 n^2}{2\pi^3 (4n^2-1)} \sum_{i=1}^N z_i e^{- 2 n^2 \pi^2 t_i} \left( e^{\pi^2 t_i}  - 1 \right) + \sum_{i=1}^N e^{-2\pi^2 t_i} + \frac{1}{2 \sigma^2}.
\end{align*}
All $N$ terms in the first summation are positive (with the exception of the $z_i$ coefficients) and decreasing with respect to $n$, so if we choose the $\{z_i\}$ sufficiently negative that $D^2 J(0)(v_q,v_q) < 0$, the Hessian will also be negative for all $v_n$ with $1 \leq n \leq q$.

To complete the proof, we choose the prior
\begin{align*}
	u_0 := \sigma^2 \Delta^{-1} p(0).
\end{align*}
It follows immediately from (\ref{eqn:EL}) that $u=0$ is a critical point of $J$.

\section*{Acknowledgments}
The author would like to thank Damon McDougall for numerous enlightening conversations throughout the preparation of this work. This research has been supported by the Office of Naval Research under the MURI grant N00014-11-1-0087.

\bibliographystyle{plain}
\bibliography{4Dvar}

\end{document}